\documentclass[letterpaper]{article}
\usepackage{amsfonts}
\usepackage{amsmath}
\usepackage{amssymb}
\usepackage{amsthm}
\usepackage{graphicx}
\usepackage{mathrsfs}
\usepackage{colonequals}
\usepackage{units}
\usepackage{euscript}

\usepackage[margin=1.50in]{geometry}

\usepackage{url}
\usepackage{xcolor}
\usepackage[colorlinks,citecolor=blue,linkcolor=red!70!black]{hyperref}
\usepackage[nameinlink]{cleveref}
\usepackage{comment}

\theoremstyle{plain}
\newtheorem{theorem}{Theorem}[section]
\newtheorem{lemma}[theorem]{Lemma}

\newtheorem{proposition}[theorem]{Proposition}

\crefname{fact}{Fact}{Facts}

\crefname{claim}{Claim}{Claims}
\newtheorem*{claim*}{Claim}

\theoremstyle{definition}

\crefname{question}{Question}{Questions}

\theoremstyle{remark}
\newtheorem{remark}[theorem]{Remark}

\makeatletter
\let\c@equation\c@theorem
\makeatother
\numberwithin{equation}{section}

\newcommand{\define}[1]{\emph{#1}}

\newcommand{\sd}[1]{}
\newcommand{\gw}[1]{}

\newcommand{\R}{\mathbb{R}}
\newcommand{\Z}{\mathbb{Z}}

\newcommand{\C}{\mathbb{C}}
\newcommand{\N}{\mathbb{N}}
\newcommand{\floor}[1]{\left\lfloor#1\right\rfloor}

\newcommand{\norm}[1]{\left\Vert#1\right\Vert}
\newcommand{\abs}[1]{\left\vert#1\right\vert}
\newcommand{\inv}{^{-1}}

\newcommand{\surf}{\Sigma} 
\newcommand{\punct}{P} 
\DeclareMathOperator{\Mod}{Mod} 
\DeclareMathOperator{\Homeo}{Homeo} 
\newcommand{\T}{{\EuScript T}} 

\newcommand{\qd}{{\EuScript Q}} 
\newcommand{\uqd}{\qd^1} 
\newcommand{\ad}{{\EuScript A}} 
\newcommand{\diff}{{\EuScript D}} 
\newcommand{\udiff}{\diff^1} 
\newcommand{\uad}{{\ad^1}} 
\newcommand{\M}{{\EuScript M}} 
\newcommand{\SL}{\mathrm{SL}} 
\newcommand{\SO}{\mathrm{SO}} 
\newcommand{\Hil}{{\EuScript H}} 
\newcommand{\hil}{\Hil}
\newcommand{\ltwo}{{\EuScript L}} 
\newcommand{\unitary}{{\sf U}} 
\newcommand{\un}{\unitary}
\newcommand{\shrink}{\mathcal{B}} 
\newcommand{\targs}{\shrink} 
\newcommand{\tar}{B} 
\newcommand{\hit}{{\EuScript E}} 
\newcommand{\eahit}{\hit^{\mathrm{a}}} 

\begin{document}

\renewcommand{\thefootnote}{\fnsymbol{footnote}} 
\renewcommand{\thefootnote}{\arabic{footnote}} 

\title{Discretely shrinking targets in moduli space}

\author{Spencer Dowdall and Grace Work
}
\date{\today}

\maketitle

\begin{abstract}
We consider the discrete shrinking target problem for Teichm\"uller geodesic flow on the moduli space of abelian or quadratic differentials and prove that the discrete geodesic trajectory of almost every differential will hit a shrinking family of targets infinitely often provided the measures of the targets are not summable. This result applies to any ergodic $\SL(2,\R)$--invariant measure and any nested family of spherical targets. Under stronger conditions on the targets, we moreover prove that almost every differential will eventually always hit the targets.
As an application, we obtain a logarithm law describing the rate at which generic discrete trajectories accumulate on a given point in moduli space. These results build on work of Kelmer \cite{Kelmer} and generalize theorems of Aimino, Nicol, and Todd \cite{aimino-nicol-todd}.
\end{abstract}

\section{Introduction} 

In a finite measure space $(X,\mu)$ with a measure preserving transformation $T\colon X\to X$, the \emph{discrete shrinking target problem} considers the question of which trajectories $\{T^n(x)\}_{n\in \N}$ will hit a shrinking family of nested measurable subsets $\tar_1\supset B_2 \supset\dotsb$ infinitely often. More precisely, defining the \emph{hitting set} for the family $\targs = \{\tar_n\}_{n\in \N}$ to be
\[\hit_T(\targs) = \Big\{x\in X \mid T^n(x)\in \tar_n\text{ for infinitely many }n\in\N\Big\},\]
the goal of the shrinking target problem is to study $\hit_T(\targs)$ and characterize, for example, when it has full or null measure. In the case of an ergodic transformation, it is easy to see that $\hit_T(\targs)$ will have full measure whenever $\mu(\tar_n)$ does not converge to $0$. On the other hand, the Borel--Cantelli Lemma implies that $\mu(\hit_T(\targs)) = 0$ whenever $\sum_{n}\mu(B_n) < \infty$. In the opposite case $\sum_{n} \mu(B_n) = \infty$, the converse to the Borel--Cantelli Lemma implies that $\hit_T(\targs)$ has full measure provided the events $\{x \mid T^n(x)\in \tar_n\}$ are independent.
There are nevertheless many naturally occurring settings in which this independence condition fails but the the hitting set $\hit_T(\targs)$ does has full measure.

\paragraph{Moduli space of differentials.} This paper considers the discrete shrinking target problem for the Teichm\"uller geodesic flow on the moduli spaces of unit area abelian or quadratic differentials. Fix a finite type surface $\surf$ (not the thrice-punctured sphere) with $\chi(\surf) < 0$ and let $\M = \M(\surf)$ denote the moduli space of Riemann surface structures on $\surf$ up to biholomorphism. An abelian differential on a Riemann surface $X\in \M$ is a holomorphic section $\omega$ of the holomorphic cotangent bundle, and a quadratic differential is a holomorphic section $q$ of its symmetric square.
We write $\diff(X)$ for the disjoint union of the vector spaces $\ad(X)$ and $\qd(X)$ of abelian and quadratic differentials on $X$, and write $\udiff(X)=\uad(X)\sqcup \uqd(X)$ for those differentials that have unit area with respect to the piecewise Euclidean structure they induce on $X$; see \S\ref{sec:quadratic_diffs} for details. 
The space $\diff = \diff(\surf)$ of all abelian or quadratic differentials forms a bundle $\diff\to \M$ over moduli space. The subbundle $\qd\to\M$ of quadratic differentials is canonically identified with the cotangent bundle $T^*(\M)$, and the restriction $\uqd\to \M$ to unit area quadratic differentials is the unit cotangent bundle for the Finsler Teichm\"uller metric on $\M$.
The group $G = \SL(2,\R)$ acts on $\diff$ and $\udiff$ via affine adjustment of the piecewise Euclidean structure; see \S\ref{sec:quadratic_diffs}.
The action of the diagonal subgroup $g_t = \left(\begin{smallmatrix}e^t & 0 \\ 0 & e^{-t}\end{smallmatrix}\right)$ then gives the \emph{Teichm\"uller geodesic flow} on $\uqd = T^*(\M)$: unit-speed Teichm\"uller geodesics in $\M$ are precisely given by projections of the paths $t \mapsto g_tq$ for $q\in \uqd$.

\paragraph{Shrinking targets.} Our main result characterizes  which shrinking targets in moduli space, or more generally abelian or quadratic differential space, are hit by almost every discrete geodesic trajectory. For the statement, a \define{spherical} subset of $\diff$ is one that is invariant under the action of the orthogonal subgroup $K = \SO(2)\le G$. 
Since $K$ preserves each fiber of $\pi\colon \udiff\to \M$, preimage sets $\pi\inv(B)\subset \udiff$ are automatically spherical.

\begin{theorem}
\label{thm:main}
Let $\mu$ be any ergodic, $G = \SL(2,\R)$--invariant, probability measure on $\udiff$. Let  $\tar_1 \supset \tar_2 \supset \dotsb$ be a nested family $\targs$ of $\mu$--measurable spherical subsets $\tar_n\subset \udiff$, and consider the hitting set $\hit_g(\targs) = \Big\{\xi\in \udiff \mid \{n\in \N \mid g_n\xi\in \tar_n\}\text{ is infinite}\Big\}$.
\begin{enumerate}
\item\label{maincase:summable} If $\sum_{n=1}^\infty \mu(\tar_n) < \infty$, then $\mu(\hit_g(\targs)) = 0$.
\item\label{maincase:inf_sum} If $\sum_{n=1}^\infty \mu(\tar_n) = \infty$, then $\hit_g(\targs)$ has full $\mu$--measure.
\item\label{maincase:unbounded} If, moreover, $\{n\mu(\tar_n)\}_{n\in \N}$ is unbounded, then there is a subsequence $n_j$ such that
\[\lim_{j \to \infty}\frac{\#\{1 \le i \le n_j \mid g_i \xi\in \tar_{n_j}\}}{n_j \mu(\tar_{n_j})} = 1\qquad\text{for $\mu$--almost every }\xi\in \udiff.\]
\end{enumerate}
\end{theorem}

\begin{remark}
There are numerous measures---such as the extensively studied Masur--Veech measures (see \S\ref{sec:invariant_measures}) or measures coming from arbitrary orbit closures---to which \Cref{thm:main} applies.
  Any ergodic measure $\mu$ is in fact supported on one of the subsets $\uad$ or $\uqd$, which are both preserved by $G$.
    We have opted to work with their disjoint union $\udiff$ simply to economize on notation.
\end{remark}

\begin{remark}
A random sequence $\xi_1,\dotsc,\xi_n$ would land in a measurable subset $B$ approximately $n\mu(B)$ times; 
thus the conclusion of \Cref{thm:main}.\ref{maincase:unbounded} says there is a subsequence of targets that are hit asymptotically the correct number of times.
\end{remark}

In \cite[Theorem 6.7]{aimino-nicol-todd}, Aimino, Nicol, and Todd consider the \emph{continuous} shrinking target problem for the Teichm\"uller geodesic and prove the continuous-time analog of \Cref{thm:main}.\ref{maincase:summable}--\ref{maincase:inf_sum} for the Masur--Veech measure on the principle stratum (see \S\ref{sec:quadratic_diffs}) in the case of targets that are preimages $\pi\inv(B_t)$ of shrinking metric balls $B_t\subset \M$.
As requiring a trajectory $\{g_t\xi\}_{t\in \R_+}$ to hit the targets at infinitely many \emph{integer} times is more subtle and restrictive than requiring $\{t\in \R_+ \mid \pi(g_t\xi)\in B_t\}$ to be unbounded, \Cref{thm:main} thus strengthens the results in \cite{aimino-nicol-todd} while also extending them to arbitrary measures  and more general families of targets.

We also consider the finer question of which differentials eventually always hit the targets. Given a shrinking family $\targs = \{\tar_n\}_{n\in \N}$ in $\udiff$, say that  $\xi\in \udiff$ \emph{hits} the target $\tar_n$ if $g_i \xi\in \tar_n$ for some $1 \le i \le n$. We then let $\eahit_g(\targs)$ denote the set of differentials that hit $\tar_n$ for all sufficiently large $n$.
If $\xi\in \eahit_g(\targs)$, then for all large $n$ we may choose $1 \le j_n\le n$ so that $g_{j_n}\xi\in \tar_n \subset \tar_{j_n}$. 
Clearly this sequence $\{j_n\}$ is unbounded unless $\xi$ lies in $g_{-k}(\cap_n \tar_n)$ for some $k\in \N$. As $\cup_k g_k(\cap_n\tar_n)$ has measure zero when $\mu(\tar_n) \to 0$ and otherwise $\hit_g(\targs)$ has full measure (\Cref{thm:main}.\ref{maincase:inf_sum}), we conclude that $\eahit_g(\targs)\subset \hit_g(\targs)$ up to a set of measure zero.
The next theorem gives conditions ensuring this smaller set has full measure and, moreover, that generic differentials hit \emph{all} small targets approximately the expected number of times. 

\begin{theorem}
\label{thm:always_hit}
Let $\mu$ and $\targs = \{\tar_n\}_{n\in\N}$ be as in \Cref{thm:main}, and set
\[\eahit_{g}(\targs) = \Big\{\xi\in \udiff \mid \{g_1 \xi, \dotsc, g_n\xi\}\cap B_n \ne \emptyset \text{ for all sufficiently large $n$}\Big\}.\]
\begin{enumerate}
\item
\label{alwayshit-fullmeas}
If there is an increasing sequence $n_1,n_2,\dotsc$ in $\N$ with $\sum_{j}(n_j\mu(\tar_{n_{j+1}}))\inv$ finite, then $\eahit_g(\targs)$ has full $\mu$--measure.
\item
\label{alwayshit-aysmptotics}
If there exists $\lambda > 1$ and an increasing sequence $n_1,n_2,\dotsc$ in $\N$ with $\sum_j (n_j \mu(\tar_{n_{j+1}}))\inv$ finite and $n_{j+1}\mu(\tar_{n_j})\le \lambda n_j \mu(\tar_{n_{j+1}})$ for all $j$, then  $\mu$--almost every $\xi\in \udiff$ satisfies
\[ \frac{1}{2\lambda} \le  \frac{\#\{1\le i \le n \mid g_i \xi\in B_n\}}{n\mu(B_n)} \le 2\lambda\qquad\text{for all sufficiently large $n$.}\]
\end{enumerate}
\end{theorem}

\begin{remark}
While the hypotheses of \Cref{thm:always_hit}.\ref{alwayshit-fullmeas}--\ref{alwayshit-aysmptotics} may seem restrictive, they in fact hold in many natural situations. For example, they are automatically satisfied (with $n_j = 2^j$) provided the the targets $\tar_n$ decay at the rate $\mu(\tar_n)\asymp n^{-\kappa}$ for some $0 < \kappa < 1$. We will use these always-hitting results to obtain the strong logarithm laws in \Cref{thm:loglaw} below.
\end{remark}

\Cref{thm:main,thm:always_hit} extend Kelmer's work on shrinking targets for discrete time flows on hyperbolic manifolds to the context of Teichm\"uller geodesic flow
and may be viewed as analogues of Theorems 1 and 2, respectively, in \cite{Kelmer}.

\paragraph{Logarithm laws.} Since the geodesic flow is ergodic, the orbit closure $\overline{\{g_n\xi \mid n\in \N\}}$ of almost every differential $\xi\in \udiff$ will equal the support of the measure. That is, the sequence $\{g_n\xi\}$ ultimately accumulates on almost every point $\xi_0$. It is natural to quantify this phenomenon by asking how long it takes for $g_n\xi$ to land within some fixed distance of $\xi_0$ or how close the orbit gets to $\xi_0$ in its first $n$ steps. Our shrinking target results imply \emph{logarithm laws} describing these quantities.

More precisely, consider the path metric $d_\M$ that Moduli space inherits from the Teichm\"uller metric on $\T$; see \S\ref{sec:teich_space}. For any Riemann surface $X\in \M$ and  differential $\xi\in \udiff$, we then let 
\[d_n(\xi,X) = \min_{0\le j \le n} d_\M(\pi(g_j \xi),X)\]
denote the closest that the projected orbit comes to $X$ in its first $n\ge 0$ steps and dually let
\[\tau_r(\xi, X)  = \inf\{n\in \N \mid d_\M(\pi(g_n \xi),X) \le r\}\]
denote the first time the orbit is within $r > 0$ of $X$. Calculating the decay/growth rate of $d_n$ and $\tau_r$ requires fine estimates for the measure of small metric balls in $\M$. Such estimates are most readily obtained (\Cref{lem:small_balls}) in the case of the Masur--Veech measure $\lambda^1_{\beta}$ on the principle stratum $\uqd(\alpha)$ where $\alpha = (-1,\dotsc,-1,1,\dotsc,1;-1)$; see \S\ref{sec:quadratic_diffs} and \S\ref{sec:ball_measures} for details.

\begin{theorem}
\label{thm:loglaw}
Fix a Riemann surface $X\in \M$ and let $\lambda_\alpha^1$ denote the Masur--Veech measure on the principle stratum $\uqd(\alpha)$ of unit-area quadratic differentials. Then for $\lambda_\alpha^1$--almost every $q\in \uqd$ one has
\[\lim_{n\to \infty} \frac{\log(d_n(q,X))}{\log(1/n)} = \frac{1}{\dim_\R(\M)}
\qquad\text{and}\qquad
\lim_{r\to 0} \frac{\log(\tau_r(q,X))}{\log(1/r)} = \dim_\R(\M).\]
\end{theorem}

We emphasize that \Cref{thm:loglaw} provides actual limits describing the decay and growth of $d_m$ and $\tau_r$, rather than merely limits superior as is typically the strongest consequence of shrinking target results as in \Cref{thm:main}. Here we obtain precise limits by employing the stronger results from \Cref{thm:always_hit}.

\begin{remark}
Applying dynamical shrinking target results to geometric questions about logarithm laws requires using measures that are related to the geometry of moduli space, since to quantitatively use \Cref{thm:main} one must know how the measure of a metric ball decays with its radius. This is the reason we obtain \Cref{thm:loglaw} only for the Masur--Veech measure $\lambda_\alpha^1$ and not for general ergodic $\SL(2,\R)$--invariant measures; see \S\ref{sec:ball_measures}.
\end{remark}

\paragraph{Historical Context.}

This work builds on a long history of shrinking target problems and dynamical Borel--Cantelli lemmas in various settings. 
Sullivan's 1982 paper \cite{Sullivan-disjointSpheres} established a logarithm law for cusp excursions of the geodesic flow on finite-volume hyperbolic manifolds;
this was later extended to the general setting of diagonal flows on homogeneous spaces by Kleinbock and Margulis \cite{KleinbockMargulis}.
These results
take horospherical cusp neighborhoods as the shrinking targets;
turning instead to precompact targets, Maucourant \cite{Maucourant} solved the continuous shrinking target problem for metric balls in hyperbolic manifolds and obtained a logarithm law describing the rate that typical geodesic trajectories approach a given point. These results were later refined and extended to discrete geodesic flows by Kleinbock and Zhao \cite{KleinbockZhao18}.
In the setting of the Teichm\"uller geodesic flow, Masur \cite{Masur-loglaw} proved a logarithm law for cusp excursions of typical geodesics in a Teichm\"uller disk, and as mentioned above, Aimino, Nicol, and Todd \cite{aimino-nicol-todd} recently solved a continuous-time shrinking target problem for the case of nested metric balls.
Their approach is combinatorial in nature and is based on Rauzy-Veech-Zorich renormalization on the space of interval exchange maps and the application of this framework to Teichm\"uller geodesic flow on the space of translation surfaces.
For more related result, see also 
\cite{Dolgopyat-limitThms}
\cite{Gadre-excursionSums}
\cite{Galotolo-dimHit}
\cite{GorodnikShah}
\cite{GuptaNicolOtt}
\cite{HaydnNicolPerssonVaientiSandro}
\cite{philipp-metricalNumThms} 
 and \cite{Athreya-loglaws} for a nice survey of this area.

Many of these results utilize exponential decay of correlations and require some sort of regularity on the targets considered. 
In our case of the Teichm\"uller geodesic flow on the moduli space of quadratic differentials, the shrinking target results  in \cite{aimino-nicol-todd} rely on the exponential mixing for H\"older observables proven by Avila--Gou\"ezel--Yoccoz \cite{AvilaGouezelYoccoz}. Indeed, the restriction to metric balls here is due in part to the necessity of using H\"older observables.

Kelmer \cite{Kelmer} recently introduced a new 
spectral theory approach that utilizes effective mean ergodic theorems (see also \cite{GhoshKelmer}) to
conclude strong shrinking target results---which apply to very general targets---for both the discrete geodesic and discrete horocyclic flows on hyperbolic manifolds. 
These ideas were subsequently extended to homogeneous spaces by Kelmer and Yu \cite{KelmerYu}.
We follow Kelmer's approach and adapt it to the Teichm\"uller flow setting by using
Eskin and Mirzakhani's result \cite{EM-invariantMeas} that every ergodic $\SL(2,\R)$--invariant probability measure $\mu$ on $\udiff$ is algebraic
and
Avila and Gou\"ezel's result \cite{AvilaGouezel} that the Laplacian operator on $L^2(\udiff,\mu)$ has a spectral gap for any such measure $\mu$.

\paragraph{Outline.}
In \S\ref{sec:prelims} we establish notation and review the necessary background material, including 
the unitary representation theory of $\SL(2,\R)$ (\S\ref{sec:special_linear_group}), and
the moduli space of differentials (\S\ref{sec:quadratic_diffs}) 
along with its period coordinates and invariant measures (\S\ref{sec:invariant_measures}). 
Additionally, in \S\ref{sec:ball_measures} we derive estimates for the measures of balls in $\M$.
Our main results are proven in \S\ref{sec:shrinking_targets}, with \S\ref{sec:hit_the_targets} and \S\ref{sec:always-hitt-targ} respectively devoted to \Cref{thm:main} and \Cref{thm:always_hit}.
The key ingredients for these proofs are established in \S\ref{sec:mean_ergodic}, where we follow Kelmer's \cite{Kelmer} approach in deriving an effective mean ergodic theorem for spherical functions in $L^2(\udiff,\mu)$, and in \S\ref{sec:quasi-indep}, where we formulate the quasi-independence needed for the converse to the Borel--Cantelli lemma. 
Finally, the logarithm laws of \Cref{thm:loglaw} are proven in \S\ref{sec:logarithm-laws}.

\paragraph{Acknowledgments.}

The authors would like to thank Jayadev Athreya and Vaibhav Gadre for several helpful conversations and their willingness to answer numerous questions concerning the topics of this paper. 
We also thank the anonymous referee for their careful reading of the paper and helpful suggestions.
The first named author was partially supported by NSF grant DMS-1711089.

\section{Preliminaries}
\label{sec:prelims}
Throughout, $G$ will denote the Lie group $G = \SL(2,\R)$ of real $2\times 2$ matrices with determinant one. 
For $t\in \R$ and $\theta\in \R$ consider the elements
\[g_t = \begin{pmatrix}e^t & 0 \\ 0 & e^{-t}\end{pmatrix} 
\quad\text{and}\quad
r_\theta = \begin{pmatrix} \cos\theta & \sin\theta \\ -\sin\theta & \cos\theta\end{pmatrix}.\]
Varying $t,\theta$, these respectively comprise the \define{diagonal} $A = \{g_t \mid t\in \R\}$ and \define{rotation} subgroups $K = \{r_\theta \mid \theta\in \R\} = \SO(2)$.
The Lie algebra of $G$ is the vector space $\mathfrak{g}$ of $2\times 2$ real matrices with trace $0$; it has a basis given by 
\begin{equation}
\label{lie_alg_basis}
W = \begin{pmatrix} 0 & 1\\-1&0\end{pmatrix},\quad
Q = \begin{pmatrix} 1 & 0 \\ 0 & -1 \end{pmatrix},
\quad\text{and}\quad
V = \begin{pmatrix}0 & 1 \\ 1 & 0\end{pmatrix}.
\end{equation}

\subsection{Unitary representation theory of $\SL(2,\R)$}
\label{sec:special_linear_group}
We briefly review the necessary representation theory of the Lie group $G = \SL(2,\R)$. For details, we direct the reader to \cite{Knapp-ss-rep-theory}. A \define{unitary representation} of $G$ is a homomorphism $\pi\colon G\to \un(\Hil)$ of $G$ into the unitary group $\un(\hil)$ of a complex Hilbert space $\hil$ such that the induced action map $G\times \hil\to \hil$ is continuous. 
Two such representations $\pi_i\colon G\to \un(\hil_i)$ are \define{unitarily equivalent} if there is a unitary isomorphism $\varphi\colon \hil_1\to\hil_2$ so that $\varphi\pi_1(g) = \pi_2(g)\varphi$ for all $g\in G$.
We refer to a representation simply by the name of its Hilbert space $\hil$ and use the notation $g\cdot v$, where $g\in G$ and $v\in \hil$, as shorthand for $\pi(g)(v)$. 
A subspace $V$ of a unitary representation $\hil$ is \define{invariant} if $g\cdot V \subset V$ for all $g\in G$; the representation is \define{irreducible} if the only closed invariant subspaces are $\{0\}$ and $\hil$. 

A vector $v$ in a unitary representation $\hil$ is a \define{$C^k$--vector}, where $k \in \N\cup\{\infty\}$, if the assignment $g\mapsto g\cdot v$ defines a $C^k$ map $G\to \hil$. The set of $C^\infty$--vectors is dense in $\hil$. 
A vector $v\in \hil$ is \define{spherical} if $k\cdot v = v$ for all $k\in K$ and is \define{$K$--finite} if $K\cdot v$ is contained in a finite-dimensional closed subspace of $\hil$.

Each Lie vector $X\in \mathfrak{g}$ determines an unbounded operator $L_X$ on any unitary representation $\hil$ defined on the subspace of $C^1$--vectors by the rule
\[ L_X (v) = \lim_{t\to 0} \frac{\exp(tX)\cdot v - v}{t}.\]
The \define{Casimir operator} of $\hil$ is then defined on the space of $C^2$--vectors as
\[\Omega = (L_W^2 - L_Q^2 - L_V^2)/4,\]
where $\{W,Q,V\}$ is the basis of $\mathfrak{g}$ specified in \eqref{lie_alg_basis}.
From the elementary calculation $\langle L_X(v),w\rangle = -\langle v, L_X(w)\rangle$ we see that $\Omega$ is symmetric: $\langle \Omega(v), w\rangle = \langle v, \Omega(w)\rangle$ for all $C^2$--vectors $v,w\in \hil$. It is known that closure of $\Omega$ is self-adjoint and that $\Omega$ commutes with $L_X$ for every $X\in \mathfrak{g}$ and with $\pi(g)$ for each $g\in G$: $\Omega(g\cdot v) = g\cdot \Omega(v)$. 
When $\hil$ is irreducible, Schur's Lemma therefore provides a scalar $\lambda = \lambda(\hil)$ such that $\Omega(v) = \lambda v$ for all $C^2$--vectors $v\in \hil$; the fact that $\Omega$ is symmetric assures that $\lambda(\hil)\in \R$. 
The \define{spectrum} of the unitary representation $\hil$ is
\[\Lambda(\hil) = \{\lambda\in \C \mid \Omega - \lambda I\text{ does not have a bounded inverse}\}\subset\R.\]

A function $G\to \C$  of the form $g\mapsto \langle g\cdot v, w\rangle$, where $v,w\in \hil$, is called a \define{matrix coefficient} of the unitary representation $\hil$. Much is known about the asymptotic behavior of matrix coefficients along the diagonal subgroup $A = \{g_t \mid t\in \R\}$, especially when $v$ and $w$ are $K$--finite; see for example \cite[VIII.\S13]{Knapp-ss-rep-theory}.
We shall only require the following basic lemma which is derived from Ratner's work \cite{Ratner-rate_of_mix}.

\begin{lemma}
\label{lem:decay_wrt_spectrum}
There exists a universal constant $C_0$ such that the following holds. 
Fix $0 < \delta < 1$ and let $\hil$ be a unitary representation of $G$ without nonzero invariant vectors and such that $\Lambda(\hil)\cap (0,\frac{1-\delta^2}{4}) = \emptyset$. 
Then for any spherical vectors $v,w\in \hil$ and $t\ge 1$ we have
\[\abs{\langle g_t\cdot v, w\rangle} \le C_0 \norm{v}\norm{w} t e^{-t(1-\delta)}.\]
 \end{lemma}
\begin{proof}
The lemma follows immediately from the quantitative version of \cite[Theorem 3]{Ratner-rate_of_mix} that Matheus gives in \cite{Matheus-quantitativeRatner}. For this, set  $\lambda = \beta(\hil) = \inf\left(\Lambda(\hil)\cap (0,\tfrac{1}{4})\right)$ (note that our Casimir operator $\Omega$ is the negative of that used in \cite{Matheus-quantitativeRatner}). The relevant terms from  \cite{Matheus-quantitativeRatner} are then the function
\[b_\hil(t) = b_\lambda(t) 
= \begin{cases}
    t e^{-t}, &  \lambda \ge \tfrac{1}{4}\\
    t e^{-(1-\sqrt{1-4\lambda})t}, & 0 < \lambda < \tfrac{1}{4}\\
    t e^{-2t}, & \lambda \le 0
  \end{cases}
\]
and the constant
\[\tilde{K}_\hil = \tilde{K}_\lambda \le \max\left\{
(1+2\sqrt{2})e + \tfrac{32 + \sqrt{2}}{3 e^3 (1-e^{-4})^2},\; 3e +e^2 + \tfrac{4}{9e^3(1-e^{-4})},\; e^2\right\}\equalscolon C_0.\]
Our hypothesis on $\Lambda(\hil)$ implies $\lambda \ge \frac{1-\delta^2}{4}$ and accordingly $-(1-\sqrt{1-4\lambda})\le -(1-\delta)$.  Since $-2 \le -1 \le -(1-\delta)$ as well, in any case we have $b_\hil(t) \le t e^{-t(1-\delta)}$.

We may now apply \cite[Theorem 2]{Matheus-quantitativeRatner} to estimate $\abs{\langle g_t \cdot v, w\rangle}$. 
Since our vectors $v,w$ are spherical, they satisfy $L_W(v) = L_W(w) = 0$ and therefore the bound reduces to
\[\abs{\langle g_t \cdot v, w\rangle} \le \tilde{K}_{\hil}\norm{v}\norm{w}b_{\hil}(t) 
\le C_0 \norm{v}\norm{w} t e^{-t(1-\delta)}. \qedhere\]
\end{proof}

\begin{remark}
Related estimates on $\abs{\langle g_t\cdot v, w\rangle}$ hold more generally but are more involved when $v,w$ are not spherical.
It is likely that our arguments could be carried out, albeit with more technical considerations, in the case of shrinking targets that are merely $K$--finite. 
\end{remark}

\subsection{Teichm\"uller space and its metric}
\label{sec:teich_space}

Throughout we let $\surf$ denote a fixed orientable closed surface $\surf$ of genus $g\ge 0$ equipped with a (potentially empty) finite set $\punct\subset \surf$ of distinguished points that we term \define{punctures}. 
Writing $p = \abs{\punct} \ge 0$, we assume that $\chi(\surf\setminus \punct)<0$ and that $(g,p) \ne (0,3)$, which is equivalent to $3g-3+p > 0$.
We write $\T = \T(\surf)$ for the \define{Teichm\"uller space} of Riemann surface structures on $\surf$ up to isotopy fixing $\punct$.
The \define{mapping class group} $\Mod(\surf)$ is the quotient 
of the group $\Homeo^+(\surf,\punct)$ of orientation-preserving homeomorphisms of $\surf$ that setwise preserve $\punct$ by the path component $\Homeo_0(\surf,\punct)$ of the identity.
The group 
of homeomorphisms naturally
acts on the set of Riemann surface structures---a homeomorphism $\phi$ sends a Riemann surface structure with atlas $\{z_\alpha\colon \surf\supset U_\alpha\to \C\}$ to the structure defined by the atlas $\{z_\alpha \circ \phi\inv\}$---and this descends to an action of $\Mod(\surf)$ on $\T$. The quotient is the \define{moduli space} $\M = \M(\surf)$ of Riemann surfaces 
 of genus $g$ with $p$ distinguished points.

The \define{Teichm\"uller distance} $d_\T$ between two Riemann surface structures $X$ and $Y$ on $\surf$ is the infimum of $\frac{1}{2}\log(K_\phi(X,Y))$ over all 
maps $\phi\in \Homeo_0(\surf,\punct)$ that are quasiconformal with respect to these structures, 
where $K_\phi(X,Y)$ denotes the quasiconformal constant of $\phi$. Topologically, $\T$ is a smooth manifold homeomorphic to $\R^{6g-6+2p}$, and there is a Finsler metric on $\T$ such that the Teichm\"uller distance $d_\T(X,Y)$ is realized as the shortest Finsler length of a path from $X$ to $Y$. This metric is invariant under the $\Mod(\surf)$ action and thus descends to a Finsler metric on $\M$ with associated path metric $d_\M$

\subsection{Notation}
Given some parameter $\delta$ (such as a point in space or a measure on that space) we use the notation $A \ll_\delta B$ or $B \gg_\delta A$ to mean that there is a constant $c > 0$, depending only on our fixed surface $\surf$ and the parameter $\delta$, such that $A \le cB$. We write $A \asymp_\delta B$ to mean $A \ll_\delta B$ and $B \ll_\delta A$. We similarly use $\ll$ and $\asymp$ when the constant $c$ depends only on $\surf$.

\subsection{Moduli space of differentials}
\label{sec:quadratic_diffs}

An \define{abelian differential} on a Riemann surface structure $X\in \T(\surf)$ is a holomorphic $1$--form, that is, a section of the holomorphic cotangent bundle of $X$. A \define{quadratic differential} is a section of the symmetric square of the meromorphic cotangent bundle whose poles are all simple and occur in the distinguished set $\punct$. 
Abelian and quadratic differentials may, respectively, be concretely represented in local coordinates $\{z_\alpha \colon U_\alpha \to \C\}$ by families of functions $h_\alpha, \phi_\alpha\colon z_\alpha(U_\alpha)\to \C$ such that on any two overlapping charts $z_\alpha$ and $z_\beta$ one has
\[h_\beta(z_\beta) \left(\frac{d z_\beta}{d z_\alpha}\right) = h_\alpha(z_\alpha)
\qquad\text{or, respectively,}\qquad
  \phi_\beta(z_\beta)\left(\frac{d z_\beta}{d z_\alpha}\right)^2 = \phi_\alpha(z_\alpha)\]
and, furthermore, the functions $h_\alpha$ are holomorphic while the functions $\phi_\alpha$ have only simple poles and are holomorphic on $\surf\setminus \punct$.
In particular, the location and multiplicities of the zeros and poles of an abelian/quadratic differential on $X$ are well-defined.
From this it is clear that each abelian differential $\omega$ or quadratic differential $q$ gives rise to a natural coordinate atlas $\{z_\alpha \colon U_\alpha \to \C\}$ for $X$ such that each $U_\alpha$ contains at most one zero and the above functions take the form $h_\alpha(z_\alpha) = z_\alpha^k$ or $\phi_\alpha(z_\alpha) = z_\alpha^k$ when $U_\alpha$ contains a zero of order $k\ge -1$ and  otherwise have the form $h_\alpha = 1$ or $\phi_\alpha = 1$. We call these the \emph{translation charts} of $\omega$ or $q$ since they define a translation surface structure on $X\setminus\{\text{zeros of $\omega$}\}$ in the abelian case, or a half-translation surface structure on $X\setminus\{\text{zeros of $q$}\}$ in the quadratic case. In this translation structure, an order--$k$ zero of $\omega$ acquires a cone angle of $2\pi(k+1)$ whereas an order--$k$ zero of $q$ acquires a cone angle of $\pi(k+2)$.

We write $\ad(X)$ and $\qd(X)$ for the $\C$--vector spaces of abelian
and quadratic differentials, respectively, on a Riemann surface $X\in \T$. 
For a tuple $\beta = (\beta_1,\dotsc,\beta_n;0)$ where $n\ge p$ and $\beta_i\ge 0$ are integers such that $\sum_i \beta_i = 2g-2$ and $\beta_i\ge 1$ for $i > p$, 
we let $\ad(X,\beta)$ denote the subset of abelian differentials on $X$ that have zeros of orders $\{\beta_1,\dotsc, \beta_p\}$ at the punctures and $n-p$ other zeros of orders $\{\beta_{p+1},\dotsc, \beta_n\}$.
Similarly for $\kappa = (\kappa_1,\dotsc, \kappa_n;\epsilon)$, where $n\ge p$, $\epsilon\in \{\pm1\}$, and $\kappa_i\ge -1$ are integers with $\sum_i \kappa_i = 4g-4$ and $\kappa_i\ge 1$ for $i> p$, we write $\qd(X,\kappa)$ for the subset of quadratic differentials with zeros of orders $\{\kappa_1,\dotsc, \kappa_p\}$ at the punctures and $n-p$ other zeros or orders $\{\kappa_{p+1},\dotsc,\kappa_n\}$ that are, if $\epsilon = 1$, or are not, if $\epsilon = -1$, the square of an abelian differential.

To economize on notation, we let $\diff(X)$ be the disjoint $\ad(X)\sqcup \qd(X)$. 
Further, for any integer vector $\alpha = (\alpha_1,\dotsc,\alpha_n;\epsilon)$ we respectively define $\diff(X,\alpha)$ to mean $\ad(X,\alpha)$ or $\qd(X,\alpha)$ in the cases that $\alpha$ satisfies the conditions on $\beta$ or $\kappa$ above.
For any such space, the notation $\uad$, $\uqd$, or $\udiff$ will indicate the corresponding subset of differentials that induce (half) translation structures with total area $1$.
As $X$ varies in $\T$ we obtain vector bundles $\ad\T$ and $\qd\T$ 
over Teichm\"uller space which, under the quotient by the action of the mapping class group, descend to bundles $\ad,\qd\to \M$
whose disjoint union $\diff = \diff(\surf)$ is stratified by the associated subsets $\diff(\alpha)$.

The bundles $\ad\T$ and $\qd\T$ are endowed with an action of the group $G = \SL(2,\R)$ of real $2\times 2$ matrices as follows: Given a differential $\xi$ with translation charts $\{z_\alpha\colon U_\alpha \to \C\}$, postcomposing by $A\in G$ defines translations charts $\{A\circ z_\alpha \}$ for a new differential $A\xi$ with the same total area and orders of zeros, but over a potentially different Riemann surface. This commutes with the action of $\Mod(\surf)$ by precomposition and so descends to an action $G\curvearrowright \diff$ that preserves all the strata $\udiff(\alpha)$ introduced above.

\subsection{Invariant measures}
\label{sec:invariant_measures}

A (quadratic or abelian) differential $\xi$ allows one to develop any path $\gamma$ in $\surf$ in the translation charts and  measure its ``total displacement'' in $\C$. In this way the differential $\xi$ defines a relative cohomology class $\Phi(\xi)\in H^1(\surf,Z(\xi);\C)$, where $Z(\xi)\subset \surf$ denotes the union of $\punct$ with the zeros of $\xi$. Choosing a symplectic basis $\{\gamma_1,\dotsc, \gamma_m\}$ allows us to view this class in coordinates as
\[\Phi(\xi)\in H^1(\surf,Z(\xi);\C)= \mathrm{Hom}(H_1(\surf,Z(\xi);\C),\C)\cong\C^m.\]
By consistently choosing symplectic bases of $H_1(\surf,Z(\xi');\C)$ for all $\xi'$ in some neighborhood of $\xi$, in this way we obtain locally defined maps
\[\Phi_\alpha\colon \diff(\alpha)\to \C^{m_\alpha}\]
on each stratum. These \define{period coordinates} are local diffeomorphisms and endow each stratum with a canonical complex affine structure. Pulling back Lebesgue measure from $\C^{m_\alpha}$ defines a canonical \define{Masur--Veech} measure $\lambda_\alpha$ on each stratum $\diff(\alpha)$ that is independent of the choice of basis.
Furthermore, $\lambda_\alpha$ induces a measure $\lambda_\alpha^1$ on the subset $\udiff(\alpha)$ of unit area differentials by the rule $\lambda_\alpha^1(U) = \lambda_\alpha(\{tU : t\in (0,1]\})$.
Since $\SL(2,\R)$ acts via area-preserving transformations of the translation charts, the Masur--Veech measures $\lambda_\alpha$ are invariant under the $\SL(2,\R)$ action.

It is a fundamental result of Masur \cite{Masur} and Veech \cite{Veech} that each measure $\lambda_\alpha^1$ has finite total mass and is ergodic for the $\SL(2,\R)$ action on $\diff$. While there are many other ergodic measures on $\diff$ (for example, coming from orbit closures $\overline{G\xi}$), the celebrated work of Eskin and Mirzakhani \cite{EM-invariantMeas} shows that every such measure is \define{affine} in the sense that it essentially arises from the Lebesgue measure of an immersed submanifold $Y\looparrowright \diff(\alpha)$ that is a affine in the period coordinates of some stratum. 

Given any probability measure $\mu$ on the space $\udiff$ of unit-area (abelian or quadratic) differentials, we may consider the Hilbert space $L^2(\udiff,\mu)$ of square-integrable functions. The $\SL(2,\R)$ action makes this into a unitary representation of $G$ in which $A\in G$ sends a function $f\in L^2(\udiff,\mu)$ to the new function $\xi \mapsto (A\cdot f)(\xi) = f(A\inv \xi)$.
We rely crucially on the following result of Avila and Gou\"ezel, which, in this formulation, itself depends on the affine structure provided by Eskin and Mirzakhani \cite{EM-invariantMeas}:

\begin{theorem}[Avila--Gou\"ezel \cite{AvilaGouezel}]
\label{thm:spectral_gap}
Let $\mu$ be any $\SL(2,\R)$--invariant, ergodic probability measure on $\udiff$. Then the Casimir operator for the unitary representation $L^2(\udiff,\mu)$ has a spectral gap. More precisely,  $\Lambda(L^2(\udiff,\mu))\cap (0,\frac{1-\delta^2}{4}) = \emptyset$ for some $0 < \delta < 1$.
\end{theorem}

\subsection{The measures of balls}
\label{sec:ball_measures}

In order to derive logarithm laws for the geodesic flow, we need to know the measure of certain geometrically significant subsets of $\udiff$. For this, we consider the principle stratum $\diff(\alpha) = \qd(\alpha)$ where $\alpha = (-1,\dotsc,-1,1,\dotsc, 1;-1)$. The associated measure $\lambda_\alpha^1$ pushes forward, under the bundle map $\pi\colon \qd\to \M$, to a measure ${\sf n} = \pi_*(\lambda_\alpha^1)$ on $\M$. The Finsler metric on $\M$ induces a \emph{Buseman volume form} that in turn gives rise to an associated measure $\nu$ on $\M$; in fact $\nu$ is the Hausdorff measure for the Teichm\"uller metric \cite{Alvarez-Thompson,Busemann}. Using Masur's result that the Teichm\"uller geodesic flow is Hamiltonian \cite{Masur-hamiltonian}, in \cite[Corollary 4.7]{ddm-stathyp} it was shown that the measures $\nu$ and ${\sf n}$ are absolutely continuous and in fact related by inequalities $c_1 {\sf n} \le \nu \le c_2 {\sf n}$ for some scalar $c_2 > c_1 > 0$. This gives rise to the following estimate:

\begin{lemma}
\label{lem:small_balls}
Fix $X\in \M$. For all sufficiently small $r >0$ the ${\sf n}$--measure of the Teichm\"uller metric ball $B_X(r) \subset \M$ of radius $r$ satisfies
\[{\sf n}(B_X(r)) \asymp_X  r^{6g-6+2n}\]
\end{lemma}
\begin{proof}
Choose a coordinate chart $\varphi\colon V\to \R^{6g-6+2n}$ on some compact set $V\subset\M$ whose interior contains $X$. 
Let $\nu^\varphi$ and $B^\varphi_X(r)$ denote the Riemannian volume and metric balls associated to the standard Riemannian metric in $\varphi$--coordinates.
Since all norms on $\R^k$ are equivalent and $V$ is compact, there is a constant $c >1$ so that the Euclidean norm in $\varphi$--coordinates agrees with the Teichm\"uller Finsler norm up to a multiple of $c$. It follows that the measures $\nu^\varphi$ and $\nu$ agree up to a multiple of $c^{6g-6+2p}$ and that for all sufficiently small $r > 0$ we have $B_X(r) \le B_X^\varphi(rc)$ and $B_X^\varphi(r) \subset B_X(rc)$. The lemma now follows by noting
\[{\sf n}(B_X(r)) \asymp \nu(B_X(r)) \asymp_c \nu^\varphi(B_X^\varphi(r)) \asymp r^{6g-6+2n}.\qedhere\]
\end{proof}

\section{Shrinking targets}
\label{sec:shrinking_targets}
For the entirety of this section, we fix an ergodic, $G$--invariant probability measure $\mu$ on the moduli space $\udiff$ of unit area abelian or quadratic differentials on $\surf$. The Hilbert space $\ltwo = L^2(\udiff,\mu)$ of square-integrable complex-valued functions is then a unitary representation of $G$ on which $A\in G$ sends $f\in \ltwo$ to the function $A\cdot f$ given by $\xi \mapsto (A\cdot f)(\xi) = f(A\inv \xi)$. For $f\in \ltwo$ we write $\norm{f}_2 = (\int f\bar{f} d\mu)^{1/2}$ for the norm of $f$ in $\ltwo$.
Let $\ltwo_0\le \ltwo$ be the closed subspace of $G$--invariant vectors, and let $\ltwo_0^\perp$ denote its orthogonal complement. 
For any $f\in \ltwo$, we let $f^0 \in \ltwo_0$ denote its orthogonal projection to $\ltwo_0$ and set $f' = f-f^0 \in \ltwo_0^\perp$.  Since the action $G\curvearrowright \udiff$ is ergodic, every $G$--invariant function is $\mu$--almost everywhere constant. Thus $\{1\}$ is a basis for $\ltwo_0$ and $f^0$ is equal in $\ltwo$ to the constant function $\langle f,1\rangle$. Observe that $f'$ is spherical if and only if $f$ is spherical.

 By \Cref{thm:spectral_gap} there exists $0 < \delta < 1$ so that the $\Lambda(\ltwo)$ and hence also $\Lambda(\ltwo_0^\perp)$ is disjoint from $(0,\frac{1-\delta^2}{4})$. 
Since $\ltwo_0^\perp$ does not have nonzero invariant vectors, if follows from \Cref{lem:decay_wrt_spectrum} that there exists a constant $\sigma > 0$ depending only on $\mu$ such that 
\begin{equation}
\label{eqn:correlation_decay}
\abs{\langle g_t \cdot f', h'\rangle } \ll \norm{f'}_2\norm{h'}_2t e^{-t\sigma}\qquad\text{for all }t\ge 1\text{ and all spherical }f,h\in \ltwo.
\end{equation}

\subsection{Effective mean ergodic theorem}
\label{sec:mean_ergodic}
Consider the average of a function $f\in \ltwo$ over the first $n\ge 1$ iterates of the geodesic flow: 
\begin{equation}
\label{eqn:avg_operator}
\beta_n(f) = \frac{1}{n}\sum_{j=1}^{n} g_{-j}\cdot f,\qquad\text{so}\qquad (\beta_n(f))(\xi) = \frac{f(g_1\xi) + \dotsb + f(g_{n} \xi)}{n}.
\end{equation}
The Birkhoff ergodic theorem implies that the images $\beta_n(f)$ of $f$  under the operators $\beta_n\colon \ltwo\to\ltwo$ pointwise almost everywhere converge to the constant function $\int f d\mu = \langle f,1\rangle$ as $n\to \infty$. When $f$ is spherical, $\{\beta_n(f)\}_n$ in fact converges in $\ltwo$ at a definite rate depending on $\norm{f}_2$ as described by the following effective mean ergodic theorem:

\begin{theorem}
\label{th:emet}
For any spherical $f\in \ltwo = L^2(\udiff,\mu)$ and $n\ge 1$, we have
\[\norm{\beta_n(f) - f^0}_2 \ll_\mu \frac{\norm{f}_2}{n^{1/2}}\]
\end{theorem}
\begin{proof} 
Since $f^0\in \ltwo_0$ is invariant under $\beta_n$, we first observe that
\[\beta_n(f) = \beta_n(f^0 + f') = \beta_n(f^0) + \beta_n(f') = f^0 + \beta_n(f').\]
Therefore $\norm{\beta_n(f) - f^0}_2 = \norm{\beta_n(f')}_2$.
To bound latter quantity, we write
\begin{align}
\label{eqn:norm_bound_average}
\norm{\beta_n(f')}_2^2  &=  \frac{1}{n^2}\sum_{j=1}^{n} \sum_{k=1}^{n} \langle g_{-j}\cdot f', g_{-k} \cdot f'\rangle
=  \frac{1}{n^2}\sum_{j=1}^{n} \sum_{k=1}^{n} \langle g_{k-j}\cdot f', f'\rangle
\end{align}
Noting that the quantity $C_i = \#\{(j,k) \mid j,k=1,\dotsc, n\text{ with }k-j = i\}$ is at most $n$ for every $i$, and using the fact $\langle g_{-i}\cdot f', f'\rangle = \overline{\langle g_i\cdot f', f'\rangle}$, equation (\ref{eqn:norm_bound_average}) gives
\begin{align*}
\norm{\beta_n(f')}_2^2  &=  \abs{\frac{1}{n^2}\sum_{j=1}^{n} \sum_{k=1}^{n} \langle g_{k-j}\cdot f', f'\rangle}
= \frac{1}{n^2}\abs{\sum_{i = 1-n}^{n-1} C_i \langle g_i\cdot f', f'\rangle}
\le \frac{2}{n}\sum_{i = 0}^{n-1} \abs{\langle g_i\cdot f', f'\rangle}.
\end{align*}
Invoking equation \eqref{eqn:correlation_decay} and using $\sum_{i=1}^\infty ie^{-i\sigma} < \infty$, we now conclude the required estimate
\[
\norm{\beta_n(f) - f^0}_2^2 = \norm{\beta_n(f')}_2^2 \ll \frac{2}{n}\norm{f'}_2^2\left(1 + \sum_{i=1}^{n-1}ie^{-i\sigma}\right) \ll_\mu \frac{\norm{f}_2^2}{n}.\qedhere
\]
\end{proof}

As we will see in \S\ref{sec:hit_the_targets} below, in the case of spherical shrinking targets $\targs = \{\tar_n\}_{n\in \N}$ with $\{n\mu(\tar_n)\}$ unbounded, the effective mean ergodic theorem provides strong results about typical hitting frequencies and the hitting sets $\hit_g(\targs)$ and $\eahit_g(\targs)$. In the absence of this strong condition on $\{n\mu(\tar_n)\}$, we must take a different approach based on independence:

\subsection{Quasi-independence}
\label{sec:quasi-indep}

Consider a family $\targs$ of spherical, $\mu$--measurable shrinking targets $\tar_1 \supset B_2 \supset \dotsb$ in $\udiff$. We are are interested in the associated sets
\[E_n = g_n\inv(\tar_n) = \{\xi\in \udiff \mid g_n \xi\in \tar_n\}.\]
Notice that $\mu(E_n) = \mu(B_n)$. As mentioned in the introduction, if the events $E_n$ were independent in the sense that $\mu(E_{n_1}\cap\dotsb \cap E_{n_k}) = \mu(E_{n_1})\dotsb \mu(E_{n_k})$ for all distinct indices $n_1,\dotsc, n_k$, the converse to the Borel--Cantelli lemma would imply $\mu(\hit_g(\targs)) = 1$. In general, this independence condition need not hold for the family $\targs$, but the spectral theory of $\ltwo =L^2(\udiff,\mu)$ guarantees that a weaker form of \emph{quasi-independence} does:

\begin{proposition}
\label{prop:quasi-indep}
Let $\tar_1 \supset \tar_2 \supset \dotsb$ be a sequence of $\mu$--measurable spherical subsets $\tar_n\subset \udiff$. For $n\ge 1$ set $E_n = g_n\inv(\tar_n)$. Then for all $N> M\ge 1$ we have
\[\sum_{m,n = M}^{N} \Big(\mu(E_m\cap E_n) - \mu(E_m)\mu(E_n)\Big) \ll_\mu \sum_{n=M}^N \mu(\tar_n).\]
\end{proposition}
\begin{proof}
For $n\ge 1$, let $f_n = \chi_{B_n}$ be the characteristic function for the $n$th target, and note that $f_n$ is spherical.
Consider $f_n^0 = \langle f_n,1\rangle = \mu(\tar_n)$ and $f_n' = f_n - f_n^0\in \ltwo_0^\perp$. We have
\begin{align*}
\norm{f_n'}_2^2 
&= \langle f_n - f_n^0, f_n - f_n^0\rangle
= \langle f_n, f_n\rangle - \mu(\tar_n)\langle 1,f_n\rangle - \mu(\tar_n)\langle f_n, 1\rangle +  \langle f_n^0,f_n^0\rangle\\
&= \mu(\tar_n) - \mu(\tar_n)^2
\le \mu(\tar_n).
\end{align*}
Observe that $(g_{-n}\cdot f_n)(\xi) = \chi_{B_n}(g_n \xi)$ for all $\xi\in \udiff$, showing that $g_{-n}\cdot f_n = \chi_{E_n}$.
Hence
\begin{align*}
\mu(E_m\cap E_n) 
&= \langle g_{-m}\cdot f_m , g_{-n}\cdot f_n\rangle
= \langle f_m^0,f_n^0\rangle + \langle g_{n-m} f'_m, f'_n\rangle\\
&= \mu(\tar_m)\mu(\tar_n) + \langle g_{n-m} f'_m, f'_n\rangle
\end{align*}
for all $m,n\ge 1$. Using equation \eqref{eqn:correlation_decay} and the fact $\mu(E_n) = \mu(\tar_n)$,
it follows that 
\begin{align*}
R_{m,n}\colonequals
\Big(\mu(E_m\cap E_n) - \mu(E_m)\mu(E_n)\Big)  \ll \sqrt{\mu(\tar_m)\mu(\tar_n)}\abs{n-m}e^{-\abs{n-m}\sigma}
\end{align*}
whenever $\abs{n-m}\ge 1$. On the other hand, $R_{n,n} = \mu(E_n)(1-\mu(E_n))\le \mu(E_n)$ for all $n$.
Since the targets are shrinking, we also have $\mu(\tar_m) \le \mu(\tar_n)$ whenever $n\le  m$. Combining these observations, we find that
\begin{align*}
\sum_{m,n = M}^{N} R_{m,n}
&= \sum_{n = M}^{N} R_{n,n} + 2\sum_{n = M}^N \sum_{i = 1}^{N - n} R_{n,n+i}\\
&\ll \sum_{n = M}^{N} \mu(E_n) + 2\sum_{n = M}^N \sum_{i = 1}^{N - n} \sqrt{\mu(\tar_n)\mu(\tar_{n+i})}i e^{-i\sigma}\\
&\le \sum_{n = M}^{N} \mu(\tar_n) + 2\sum_{n = M}^N \mu(\tar_n)\sum_{i = 1}^{N - n} i e^{-i\sigma}\\
&\le \sum_{n = M}^{N} \mu(\tar_n)\left(1 + 2\sum_{i = 1}^\infty i e^{-i\sigma}\right).\qedhere
\end{align*}
\end{proof}

\subsection{Hitting the targets}
\label{sec:hit_the_targets}
\label{sec:proof_main_theorem}

We now use the estimates from \S\S\ref{sec:mean_ergodic}--\ref{sec:quasi-indep} to obtain our main result on shrinking targets. Let $\targs$ be any family of spherical $\mu$--measurable shrinking targets $\tar_1 \supset \tar_2 \supset\dotsb$ in $\udiff$, where $\mu$ remains the arbitrary $G$--invariant measure specified at the start of \S\ref{sec:shrinking_targets}. Set $E_n = g_n\inv(\tar_n)$ and consider the associated hitting set for $\targs$:
\[\hit_g(\targs) = \big\{\xi\in \udiff \mid \{n\in \N \mid g_n \xi\in \tar_n\}\text{ is infinite}\big\} = \bigcap_{k = 1}^\infty \bigcup _{n= k}^\infty E_n.\]

\begin{proof}[Proof of \Cref{thm:main}.\ref{maincase:summable}]
Suppose $\sum_{n=1}^\infty \mu(\tar_n) < \infty$. Then we have
\[\mu(\hit_g(\targs)) = \mu\left(\bigcap_{k=1}^\infty \bigcup_{n=k}^\infty E_n\right) 
\le \inf_{k\ge 1} \mu\left(\bigcup_{n=k}^\infty E_n\right) \le \inf_{k\ge 1}\sum_{n=k}^\infty \mu(\tar_n) = 0.\qedhere\]
\end{proof}

\begin{proof}[Proof of \Cref{thm:main}.\ref{maincase:inf_sum}]
Suppose that $\sum_{n=1}^\infty \mu(\tar_n) = \infty$. It has long been known, going back to the work of Schmidt, that the conclusion $\mu(\hit_g(\targs)) = 1$ may be deduced from this hypothesis and \Cref{prop:quasi-indep}; for example, \cite[Chapter 1, Lemma 10]{Sprindzuk} gives an even stronger conclusion. We include a brief proof here for completeness. Let $\Phi_N(x) = \sum_{n=1}^N \chi_{E_n}$ and $S_N = \int \Phi_N d\mu = \sum_{n=1}^N \mu(\tar_n)$. Hence $\xi\in \hit_g(\targs)$ iff $\{\Phi_N(\xi)\}$ is unbounded. Using \Cref{prop:quasi-indep}, we see that
\begin{align*}
\int (\Phi_N(\xi) - S_N)^2 d\mu(\xi)
&= \int \left(\sum_{m,n=1}^N \chi_{E_m}\chi_{E_n} - 2\mu(E_m) \chi_{E_n} + \mu(E_m)\mu(E_n)\right) d\mu\\
&= \sum_{m,n=1}^N \Big(\mu(E_m\cap E_n) - \mu(E_m)\mu(E_n)\Big)
\ll_\mu \sum_{n=1}^N \mu(\tar_n) = S_N.
\end{align*}
Therefore $\norm{\Phi_N - S_N}_2 \ll_\mu \sqrt{S_N}$. Setting $\Psi_N = \frac{1}{S_N}\Phi_N$, it follows that
\[\norm{\Psi_N -1 }_2 \ll_\mu \frac{1}{\sqrt{S_N}}.\]
Since $S_N \to \infty$ by assumption, we conclude that $\Psi_N$ converges in $\ltwo$ to the constant function $1$. After passing to a subsequence $N_i$, it follows that $\Psi_{N_i}(\xi) \to 1$ for $\mu$--almost-every $\xi\in \udiff$ (see, for example, \cite[Corollary 2.32]{Folland}). But by definition this implies $\Phi_{N_i}(\xi) \to \infty$ and hence $\xi\in \hit_g(\targs)$. Therefore $\hit_g(\targs)$ has full measure, as claimed.
\end{proof}

\begin{proof}[Proof of \Cref{thm:main}.\ref{maincase:unbounded}]
Suppose the sequence $\{n \mu(\tar_n)\}_{n\in \N}$ is unbounded. Let $f_n$ be the normalized characteristic function $f_n = \frac{1}{\mu(\tar_n)}\chi_{\tar_n}$. For the averaging operator $\beta_n$ from \eqref{eqn:avg_operator}, we have that
\[\beta_n(f_n)(\xi) = \frac{1}{n} \sum_{i=1}^{n} f_n(g_i \xi) = \frac{\#\{1\le i \le n \mid g_i \xi\in \tar_n\}}{n\mu(\tar_n)}.\]
Since $f_n^0=\langle f_n, 1\rangle = 1$ and $\norm{f_n}_2 = \frac{1}{\mu(\tar_n)}\norm{\chi_{\tar_n}}_2 = 1/\sqrt{\mu(\tar_n)}$, \Cref{th:emet} implies that
\[\norm{\beta_n(f_n) - 1}_2  \ll_\mu \frac{\norm{f_n}_2}{n^{1/2}} = \frac{1}{\sqrt{n\mu(\tar_n)}}.\]
Choosing a subsequence $n_j$ so that $n_j \mu(\tar_{n_j}) \to \infty$, it follows that $\beta_{n_j}(f_{n_j})$ converges to the constant function $1$ in $\ltwo$. Passing to a further subsequence, for $\mu$--almost-every $\xi\in \udiff$ we conclude that
\[\lim_{j \to \infty} \frac{\#\{1 \le i \le n_j \mid g_i \xi \in \tar_{n_j}\}}{n_j \mu(\tar_{n_j})} = \lim_{j \to \infty} \beta_{n_j}(f_{n_j})(\xi) = 1.\qedhere\]
\end{proof}

\subsection{Always hitting the targets}
\label{sec:always-hitt-targ}
We next turn to \Cref{thm:always_hit} and the question of when  differentials will eventually always hit the targets. Again, fix a shrinking family $\targs = \{\tar_n\}_{n\in \N}$ of $\mu$--measurable spherical targets. Let $f_n = \chi_{\tar_n}$ be the associated characteristic functions. For $n,m\in \N$, let
\[W_{m,n} = \big\{\xi\in \udiff \mid \{g_1 \xi,\dotsc,g_m \xi\}\cap \tar_n = \emptyset\big\}\]
be the set of differentials whose first $m$ iterates miss the target $\tar_n$. Then 
\begin{equation}
\label{eqn:always-hitting-set-and-W}
\eahit_g(\targs) = \{\xi\in \udiff\mid \xi\notin W_{n,n}\text{ for all sufficiently large }n\}
= \bigcup_{k=1}^\infty \bigcap_{n=k}^\infty (\udiff -  W_{n,n}).
\end{equation}

We know from \Cref{th:emet} that $\beta_m(f_n)$ converges to  the constant function $\mu(B_n) = \langle f_n, 1\rangle$ as $m\to \infty$. To analyze those differentials for which this convergence is poor, for each $\kappa > 1$ and $m,n\in \N$ let us set
\[Y^\kappa_{m,n} = \left\{\xi \in \udiff \;\Big\vert\;  \frac{\#\{1\le i\le m \mid g_i\xi\in B_n\}}{m} = \beta_m(f_n)(\xi)  \notin [\kappa\inv \mu(B_n) , \kappa \mu(B_n)]\right\}.\]
Now, for each $\xi\in W_{m,n}$ we have $\beta_m(f_n)(\xi) = 0$ by definition of $\beta_m$ and $W_{m,n}$. Therefore
\[\norm{\beta_m(f_n) - \mu(\tar_n)}_2^2 \ge \int_{W_{m,n}} \abs{\beta_m(f_n) - \mu(\tar_n)}^2 d\mu = \mu(W_{m,n})\mu(\tar_n)^2.\]
For each $\xi\in Y^\kappa_{m,n}$ we similarly have $\abs{\beta_m(f_n)(\xi) - \mu(B_n)} \ge (\kappa-1)\kappa\inv \mu(B_n)$ and hence
\[\norm{\beta_m(f_n) - \mu(B_n)}_2^2 \ge \int_{Y^\kappa_{m,n}} \abs{\beta_m(f_n) - \mu(B_n)}^2 \ge \mu(Y^\kappa_{m,n})\frac{(\kappa-1)^2}{\kappa^2}\mu(B_n)^2.\]
Since $\langle f_n,1\rangle = \mu(\tar_n) = \norm{f_n}_2^2$, the effective mean ergodic theorem (\Cref{th:emet}) now gives
\begin{equation}
\label{eqn:bound_W_and_Y}
\begin{aligned}
\mu(W_{m,n}) &\le \frac{\norm{\beta_m(f_n)- \mu(\tar_n)}_2^2}{\mu(\tar_n)^2} \ll_\mu \frac{\norm{f_n}_2^2}{m\mu(\tar_n)^2} = \frac{1}{m\mu(\tar_n)},\quad\text{and}\\
\mu(Y^\kappa_{m,n}) &\le \frac{\kappa^2\norm{\beta_m(f_n) - \mu(\tar_n)}_2^2}{(\kappa-1)^2 \mu(B_n)^2} \ll_\mu \frac{\kappa^2}{m(\kappa-1)^2\mu(\tar_n)}.
\end{aligned}
\end{equation}

\begin{proof}[Proof of \Cref{thm:always_hit}.\ref{alwayshit-fullmeas}]
Suppose $n_j$ is an increasing sequence in $\N$ with $\{(n_j \mu(\tar_{n_{j+1}}))\inv\}$  summable. Notice that whenever $n_{j} \le n \le n_{j+1}$, we have $W_{n,n} \subset W_{n_{j}, n_{j+1}}$ by definition of $W_{n,m}$ and the fact $\tar_{n_{j+1}} \subset \tar_n$. Consequently, using \eqref{eqn:always-hitting-set-and-W} we may write
\[\udiff - \eahit_g(\targs)
= \bigcap_{k=1}^\infty \bigcup_{n=k}^\infty W_{n,n} \subset \bigcap_{k=1}^\infty \bigcup_{n_j \ge k} W_{n_{j},n_{j+1}}.\]
The always hitting set $\eahit_g(\targs)$ therefore has full measure since by \eqref{eqn:bound_W_and_Y} its complement has 
\[\mu\big(\udiff - \eahit_{g}(\targs)\big) \le \inf_{k\ge 1}\left( \sum_{n_j \ge k} \mu(W_{n_{j},n_{j+1}})\right) \le \inf_{k\ge 1} \left(\sum_{n_j\ge k}\frac{1}{n_{j}\mu(\tar_{n_{j+1}})}\right) = 0.\qedhere\]
\end{proof}

\begin{proof}[Proof of \Cref{thm:always_hit}.\ref{alwayshit-aysmptotics}]
Suppose an increasing sequence $n_j$ in $\N$ with $\{(n_j \mu(\tar_{n_{j+1}}))\inv\}$  summable satisfies $n_{j+1}\mu(\tar_{n_j}) \le \lambda n_{j}\mu(\tar_{n_{j+1}})$ for some $\lambda > 1$. For any $\kappa > 1$, let 
\[Y^\kappa = \bigcap_{k=1}^\infty \bigcup_{n = k}^\infty \Big\{\xi\in \udiff \mid \beta_n(f_n)(\xi) \notin [\kappa\inv\mu(\tar_n), \kappa\mu(\tar_n)]\Big\} = \bigcap_{k=1}^\infty \bigcup_{n=k}^\infty Y^\kappa_{n,n}\]
denote the set of differentials $\xi$ such that for each $k\in \N$ there is some $n\ge k$ so that $\beta_n(f_n)(\xi)$ lies outside $[\kappa\inv\mu(\tar_n), \kappa\mu(\tar_n)]$. Then for each $\xi\notin Y^\kappa$ there exists some $K = K(\xi)$ such that for all $n\ge K$ we have
\[\frac{\#\{1 \le i \le n \mid g_i\xi \in \tar_n\}}{n \mu(\tar_n)} = \frac{\beta_n(f_n)(\xi)}{\mu(\tar_n)} \in [ \kappa\inv, \kappa].\]
Hence to prove the theorem it suffices to show $\mu(Y^{2\lambda}) = 0$.

For this, observe that for each $n_{j} \le n \le n_{j+1}$, the averaged function $\beta_n(f_n)$  satisfies
\[ \frac{n_{j}}{n_{j+1}} \beta_{n_{j}}(f_{n_{j+1}}) \le \frac{n_{j}}{n} \beta_{n_{j}}(f_n) \le \beta_{n}(f_n) \le \frac{n_{j+1}}{n}\beta_{n_{j+1}}(f_n) \le  \frac{n_{j+1}}{n_{j}} \beta_{n_{j+1}}(f_{n_{j}})\]
by the definition of $\beta_m$ and the fact that $0\le f_{n_{j+1}} \le f_n \le f_{n_{j}}$. Thus when  $\beta_n(f_n)(\xi) > \kappa\mu(\tar_n)$ we find that
\[\beta_{n_{j+1}}(f_{n_{j}})(\xi) \ge \frac{n_{j}}{n_{j+1}} \beta_n(f_n)(\xi)
> \frac{n_{j}}{n_{j+1}} \kappa \mu(\tar_n) 
\ge \frac{\kappa}{\lambda} \mu(\tar_{n_{j}})
\]
(here we have used our hypothesis  $\mu(\tar_n) \ge \mu(\tar_{n_{j+1}}) \ge \frac{n_{j+1}}{\lambda n_{j}}\mu(\tar_{n_{j}})$). On the other hand, when $\beta_n(f_n)(\xi) < \kappa\inv\mu(\tar_n)$ we have
\[\beta_{n_{j}}(f_{n_{j+1}})(\xi) 
\le \frac{n_{j+1}}{n_{j}}\beta_n(f_n)(\xi)
< \frac{n_{j+1}}{n_{j}\kappa}\mu(\tar_n)
\le \frac{\lambda}{\kappa} \mu(\tar_{n_{j+1}}). 
\]

Setting $\kappa = 2\lambda > 1$, it follows that $Y^{2\lambda}_{n,n}\subset Y^{2}_{n_{j},n_{j+1}}\cup Y^{2}_{n_{j+1},n_j}$ whenever $n_{j} \le n \le n_{j+1}$. 
Since the numbers $n_j$ are increasing and the targets $\tar_{n_{j}}$ are shrinking, we have that $n_j \mu(\tar_{n_{j+1}}) \le n_{j+1}\mu(\tar_{n_j})$ for all $j\in \N$. Therefore our hypothesis implies the sequences $\{(n_{j}\mu(\tar_{n_{j+1}}))\inv\}_{j\in\N}$ and $\{(n_{j + 1}\mu(\tar_{n_j}))\inv\}_{j\in\N}$ are \emph{both} summable.
Using these observations together with \eqref{eqn:bound_W_and_Y}, we now conclude the desired estimate
\begin{align*}
\mu(Y^{2\lambda})
&= \mu\left(\bigcap_{k=1}^\infty \bigcup_{n=k}^\infty Y^{2\lambda}_{n,n}\right)
\le \mu\left(\bigcap_{k=1}^\infty \bigcup_{n_j\ge k} Y^{2}_{n_{j},n_{j+1}}\cup Y^2_{n_{j+1},n_j}\right)\\
&\le \inf_{k\ge 1} \sum_{n_j \ge k}  \left(\frac{4}{n_j \mu(\tar_{n_{j+1}})} + \frac{4}{n_{j+1} \mu(\tar_{n_j})}\right) = 0.\qedhere
\end{align*}
\end{proof}

\section{Logarithm Laws}
\label{sec:logarithm-laws}
Here we explicitly consider the Masur--Veech measure $\lambda^1_\alpha$ on the principle stratum $\uqd(\alpha)=\udiff(\alpha)$ where $\alpha = (-1,\dotsc,-1,1,\dotsc,1;-1)$. Recall from \S\ref{sec:ball_measures} that we have used ${\sf n} = \pi_*(\lambda_\alpha^1)$ to denote the push-forward of measure under $\pi\colon \uqd\to\M$. Recall also that $d_\M(X,Y)$ denotes the distance between points $X$ and $Y$ in the moduli space $\M$ and that for $X\in \M$ and $\xi\in \udiff$ we have defined
\[d_n(\xi,X) = \min_{0\le j\le n} d_\M(\pi(g_j \xi),X)\quad\text{and}\quad
\tau_r(\xi,X) = \inf\{n\in \N \mid d_\M(\pi(g_n\xi),X) < r\}.\]

\begin{proof}[Proof of \Cref{thm:loglaw}]
For $\epsilon > 0$ fixed, set $r_m^\pm = m^{ \frac{-(1\pm \epsilon)}{6g-6+2p}}$ and let $\tar_m^{\pm} = \pi\inv(B_{X}(r_m^\pm))\subset \uqd$ be the preimage of the radius $r_m^\pm$ metric ball about $X$. Consider the shrinking families $\shrink^\pm = \{\tar_m^\pm\}_{m\in\N}$. By \Cref{lem:small_balls} we have
\[\lambda_\alpha^1(\tar_m^\pm) = {\sf n}(B_{X}(r_m^\pm)) \asymp_{X} (r_m^\pm)^{6g-6+2p} = \frac{1}{m^{1\pm \epsilon}}.\]
Therefore $\sum_{m=1}^\infty \lambda_\alpha^1(\tar_m^+)  < \infty$ and \Cref{thm:main}.\ref{maincase:summable} implies that $\lambda_\alpha^1(\hit_g(\shrink^+)) = 0$. The set $Z$ of $\xi\in \uqd$ such that $\pi(g_k \xi) = X$ of some $k\in \N$ also has measure zero. We claim that each $\xi\notin \hit_g(\shrink^+)\cup Z$ satisfies
\begin{equation}
\label{eqn:log_lower_bound}
\frac{\log(d_m(\xi,X))}{\log(m)} \ge \frac{-1-\epsilon}{6g-6+2p}\qquad\text{for all sufficiently large $m$}.
\end{equation}
Indeed if this is not the case we may choose indices $m_j$ tending to infinity and associated indices $k_j\le m_j$ such that 
\[0 < d_\M(\pi(g_{k_j} \xi),X) = \min_{k\le m_j} d(\pi(g_k \xi),X) =  d_{m_j}(\xi,X) < {m_j}^{\frac{-1-\epsilon}{6g-6+2p}} = r_{m_j}^+ \le r_{k_j}^+.\]
These inequalities imply that $g_{k_j}\xi\in \tar_{k_j}^+$ for each $j$ while also forcing (since $r_{m_j}^+\to 0$) the set $\{k_j \mid j\in \N\}$ to be infinite. However, this contradicts $\xi\notin \hit_g(\shrink^+)$, proving the claim.

For $\shrink^-$, on the other hand, \Cref{thm:always_hit}.\ref{alwayshit-fullmeas} implies that $\eahit_g(\shrink^-)$ has full measure since 
\[\sum_{j=1}^\infty \left(2^j \lambda_\alpha^1(\tar_{2^{j+1}}^-)\right)\inv 
\asymp_X \sum_{j=1}^\infty \frac{(2^{j+1})^{1-\epsilon}}{2^{j}}
= 2\sum_{j=1}^\infty (2^{-\epsilon})^{j+1} < \infty.
\]
For each $\xi\in \eahit_g(\shrink^-)$ we know that $\{g_1\xi,\dotsc, g_m \xi\}\cap \tar_m^-$ is nonempty for all sufficiently large $m$. But this is equivalent to saying
\begin{equation}
\label{eqn:log_upper_bound}
d_m(\xi,X) =  \min_{k\le m} d_\M(\pi(g_k\xi),X) < r_m^- = m^{\frac{-1+\epsilon}{6g-6+2p}}\qquad\text{ for all sufficiently large $m$}.
\end{equation}

Since \eqref{eqn:log_lower_bound} and \eqref{eqn:log_upper_bound} each hold for a full measure set, we conclude that
\[ \frac{-1-\epsilon}{6g-6+2p} 
\le \liminf_{m\to\infty} \frac{\log(d_m(\xi,X))}{\log(m)} 
\le \limsup_{m\to\infty} \frac{\log(d_m(\xi,X))}{\log(m)} 
\le \frac{-1+\epsilon}{6g-6+2p}\]
for $\lambda_\alpha^1$--almost every $\xi\in \uqd$. Since $\epsilon > 0$ was arbitrary, the first claim follows.

For notational convenience, for $y\in \R_+$, let us now define
\[d'_y(\xi,X) = d_{\floor{y}}(\xi,X) = \min\{d_\M(\pi(g_i\xi,X)) \mid i\in \Z\text{ with }0\le i \le y\}.\]
Notice that $\lim_{y\to\infty} \frac{\log(d'_y(\xi,X))}{\log(y)} = \lim_{m\to \infty} \frac{\log(d_m(\xi,X))}{\log(m)}$ since $\frac{\log y}{\log\floor{y}} \to 1$. Also observe from the definition of $\tau_r(\xi,X)$ that
\begin{equation}
\label{eqn:duality}
\tau_r(\xi,X) \le y \iff d'_y(\xi,X) < r
\qquad\text{and}\qquad
\tau_r(\xi,X) > y \iff d'_y(\xi,X)\ge r.
\end{equation}
By the first result, there is a $\lambda_\alpha^1$--full measure set of $\xi\in \uqd$ such that for any $\epsilon > 0$ we have
\[\frac{-(1+\epsilon)}{6g-6+2p} \le \frac{\log(d'_y(\xi,X))}{\log(y)} < \frac{-(1-\epsilon)}{6g-6+2p} \qquad\text{ for all sufficiently large $y$}.\]
Applying this to $y = r^{\frac{6g-6+2p}{-(1\pm \epsilon)}}$ (so that $r = y^{\frac{-(1\pm \epsilon)}{6g-6+2p}}$) and invoking \eqref{eqn:duality}, it follows that
\begin{align*}
r^{\frac{6g-6+2p}{-(1+\epsilon)}}<\tau_r(\xi,X) \le r^{\frac{6g-6+2p}{-(1-\epsilon)}}\qquad\text{ for all sufficiently small $r > 0$}.
\end{align*}
Rearranging, we conclude that there is a $\lambda_\alpha^1$--full measure set of $\xi\in \uqd$ for which
\[ \frac{6g-6+2p}{1+\epsilon}< \frac{\log(\tau_r(\xi,X))}{\log(1/r)} \le \frac{6g-6+2p}{1-\epsilon} \qquad\text{ holds for all sufficiently small $r > 0$.}\]
Since $\epsilon > 0$ is arbitrary, the result follows.
\end{proof}

\bibliographystyle{alphanum}
\bibliography{shrink_targets}

\bigskip

\noindent
\begin{minipage}{.55\linewidth}
Vanderbilt University\\
Department of Mathematics\\
1326 Stevenson Center\\
Nashville, TN 37240, USA\\
E-mail: {\tt spencer.dowdall@vanderbilt.edu}
\end{minipage}
\begin{minipage}{.45\linewidth}
University of Wisconsin\\
Department of Mathematics\\
480 Lincoln Drive\\
213 Van Vleck Hall\\
Madison, WI 53706\\
E-mail: {\tt work2@wisc.edu}
\end{minipage}

\end{document}